\documentclass{amsart}
\usepackage[utf8]{inputenc}
\usepackage{graphicx,amssymb,latexsym}
\input xy
\xyoption{all}
\usepackage[all]{xy}
\graphicspath{ {./images/} }

\usepackage{color}
\usepackage{pb-diagram}
\usepackage{epstopdf}
\usepackage{multicol}

\newtheorem{thm}{Theorem}[section]

\newtheorem{cor}[thm]{Corollary}
\newtheorem{lem}[thm]{Lemma}
\newtheorem{prop}[thm]{Proposition}
\theoremstyle{definition}

\newtheorem{rem}[thm]{Remark}

\numberwithin{equation}{section}

\newcommand{\ZZ}{\mathbb Z}

\newcommand{\PP}{\mathbb P}
\newcommand{\FF}{\mathbb F}

\newcommand{\lra}{\longrightarrow}

\newcommand{\ra}{\rightarrow}

\newcommand{\cA}{\mathcal{A}}

\newcommand{\cJ}{\mathcal{J}}

\newcommand{\tC}{\widetilde{C}}

\newcommand{\tg}{\Tilde{\gamma}}

\newcommand{\cM}{\mathcal{M}}

\newcommand{\cO}{\mathcal{O}}

\newcommand{\cR}{\mathcal{R}}
\newcommand{\cRH}{\mathcal{RH}}

\DeclareMathOperator{\Aut}{{Aut}}

 \DeclareMathOperator{\Ker}{ker}
  \DeclareMathOperator{\Fix}{Fix}
\DeclareMathOperator{\Pic}{Pic}
 \DeclareMathOperator{\Nm}{{Nm}}

\DeclareMathOperator{\Spec}{Spec}

 \DeclareMathOperator{\im}{Im}
 \DeclareMathOperator{\Div}{{Div}}

\DeclareMathOperator{\id}{id}

\newcommand{\s}{\sigma}
\renewcommand{\t}{\tau}
\newcommand{\ita}{\iota\tau}
\newcommand{\e}{\eta}
\newcommand{\x}{\xi}
\newcommand{\al}{\alpha}

\title{Involutions on hyperelliptic curves and Prym maps}
\author{Pawe\l{} Bor\'owka, Angela Ortega}

\address{P. Bor\'owka\\ Institute of Mathematics, Jagiellonian University in Krak\'ow\\ ul. prof. Stanisława Łojasiewicza 6, 
30-348 Kraków, Poland}
\email{pawel.borowka@uj.edu.pl}
\address{A. Ortega \\ Institut f\" ur Mathematik, Humboldt Universit\"at zu Berlin \\ Unter den Linden 6, 10099 Berlin, Germany}
\email{ortega@math.hu-berlin.de}


\begin{document}
\begin{abstract}
    We investigate the geometry of smooth hyperelliptic curves that possess additional involutions, especially from the point of view of the Prym theory. Our main result is the injectivity of the Prym map for hyperelliptic $\ZZ_2^2$-coverings over hyperelliptic curves of positive genus.   
\end{abstract} 
\maketitle

\section{Introduction}
A smooth complex hyperelliptic curve $C$ is a Riemann surface of genus $g>1$ that is a double covering of the Riemann sphere $\PP^1$. Having such a map makes hyperelliptic curves distinguishable and more accessible in many aspects since, for example, they can be described by an equation of the form $y^2=F(x)$ and in this way one can see an hyperelliptic curve as a subvariety of a weighted projective plane.

 A covering $f:C'\ra C$ will be called {\it hyperelliptic} if {\it both} curves are hyperelliptic. Such assumption is actually quite strong: if $f$ is cyclic and unbranched then 
 $\deg(f)=2$ (see \cite{O},\cite{Ries}). If $f$ is a hyperelliptic double covering, then the number of branch points has to be at most $4$ and there are constrains on a line bundle that defines the covering (see Section \ref{sec:higher} for details). On the other hand, surprisingly, a non-Galois \'etale triple covering of a genus 2 curve is hyperelliptic (see \cite{LOtriple}).

The Prym theory investigates the (connected component of) kernel of the norm map $\Nm_f:JC'\ra JC$ that can also be seen as a complementary abelian subvariety to the image of Jacobian
$f^*(JC)$ inside $JC'$ and is called the Prym variety of the covering. One can then consider the Prym map that assigns  to a covering its Prym variety.

The Prym map restricted to the locus of hyperelliptic double coverings is never injective (see the bigonal construction, \cite{NO}, or Corollaries \ref{never2} and \ref{never}). Motivated by this fact, we investigate the Prym map of hyperelliptic Klein coverings, i.e. $4:1$ Galois coverings with Galois group isomorphic to the Klein group $\ZZ_2^2$
and both curves are hyperelliptic.
In \cite{BOklein} we have shown the injectivity of 
the Prym map for the special case of  \'etale coverings 
over a  genus 2 curve. Now, we are able to show the injectivity of the hyperelliptic Prym map in full generality (any genus and including ramified coverings). We show
in Theorems \ref{injectiveKlein}, \ref{mixed_Prym}, \ref{brached_inj_Prym} and \ref{mixed4g+3} the following: 
\bigskip
\begin{thm}
Let  $\mathcal{RH}_{g,b}$ be the moduli space of hyperelliptic  Klein coverings  over a curve of genus $g>1$ which are simply ramified in $b$ points. We also include the cases $g=1,b=8$ and $g=1, b=12$. Then the corresponding Prym maps on  $\mathcal{RH}_{g,b}$
for $b \in \{ 0,4, 8, 12\} $ are (globally) injective. 
\end{thm}
\bigskip

The proofs of these theorems are based on geometric characterizations of such coverings and the description of the 2-torsion points of the involved Jacobians in terms of the Weierstrass points. In all the cases we construct an explicit inverse of the Prym map.

It has been shown that the Prym map of double coverings branched in at least 6 points (hence not hyperelliptic) is globally injective (\cite{NO}). Since hyperelliptic coverings make the bound on the number of branched points sharp, one may believe that our result is an important step in showing global injectivity of Klein Prym maps (both \'etale and branched).

The paper is organised as follows: Section 2 contains the necessary basic facts about involutions on hyperelliptic curves, following the top-down perspective.  
In Section 3 we recall the constructions for double hyperelliptic coverings to see the bottom-up perspective. Having both perspectives gives us a possibility to show what kind of data is needed to set up a Klein covering construction.

In Section 4, we generalise results from \cite{BOklein}, i.e. we prove the injectivity of the Prym map for \'etale hyperelliptic Klein coverings of any genus and we also prove the so-called mixed case, i.e. coverings ramified in 8 points.

In Section 5, we show the injectivity of the Prym map for $\ZZ_2^2$ hyperelliptic coverings branched in 12 points and another mixed case, namely coverings branched in 4 points.
The Figures 1-4 appearing in this article have been produced using the software {\it Inkscape}.

\subsection*{Acknowledgements}
The first author has been supported by the Polish National Science Center project number 2019/35/D/ST1/02385.
The second author warmly thanks for the hospitality during her stay at the Jagiellonian University in Krak\'ow, where part of this work was carried out. Her visit in Krak\'ow was covered by International Cooperation Funding Programme at the Faculty of Mathematics and Computer Science of the Jagiellonian University under the Excellence Initiative at
the Jagiellonian University.

\section{Preliminaries}
In this section we describe the geometry of hyperelliptic curves  that contain at least one more involution and its corresponding covering from the top-down perspective. Some of the results can already be found in \cite{GS} that is devoted to hyperelliptic curves with extra involutions. 

We start by recalling some basic facts about involutions on hyperelliptic curves. Let $H$ be a hyperelliptic curve of genus $g(H)=g$.
For simplicity, the hyperelliptic involution will always be denoted by $\iota$ (or $\iota_H$ if it is important to remember the curve). Let $W=\{w_1,\ldots,w_{2g+2}\}$ denote the set of Weierstrass points, which is the same as the set of ramification points of the hyperelliptic covering.
The following propositions are well-known facts (see for example \cite{Sh}). 
\begin{prop}\label{prop:permute_Wpoints}
Let $H$ be a hyperelliptic curve and $\iota$ the hyperelliptic involution. Then $\iota$ commutes with any automorphism of $H$. Every automorphism of $H$ is a lift of an automorphism of $\PP^1$ and it restricts to a permutation of $W$. In particular, if $\ZZ_2^n\subset\Aut(H)$ then $n\leq 3$.
\end{prop}

\begin{prop}\label{prop:involutions}
Let $\t\in\Aut(H)$ be a (non-hyperelliptic) involution on $H$. By $H_\t=H/\t$ we denote the quotient curve. If $g(H)=2k$ then both $\t$ and $\ita$ have exactly 2 fixed points and $g(H_\t)=g(H_{\ita})=k$. If $g(H)=2k+1$ then either $\t$ is fixed point free and $\ita$ has 4 fixed points or $\t$ has 4 fixed points and $\ita$ is fixed point free.
\end{prop}

In order to make statements easier and more compact we abuse the notation by saying that a genus 1 curve with a chosen double covering of $\PP^1$ is called hyperelliptic.

\begin{cor}\label{hypquot}
With the notation from Proposition \ref{prop:involutions}, the curves $H_\t$ and $H_{\ita}$ are hyperelliptic whose  hyperelliptic involution lifts to the involution $\iota$ on $H$.
\end{cor}

 Assume there are two involutions $\s,\t\in \Aut(H)$ such that $\s\t=\t\s$, (i.e., $\langle \s,\t \rangle \simeq \ZZ_2^2$). In such a case, the covering $H\to H/\langle \s,\t\rangle$ will be Galois with the deck group isomorphic to the Klein four-group.
Since we are interested in Prym maps, we make another natural assumption, namely $\iota\notin \langle \s,\t\rangle$, hence $g(H/\langle \s,\t\rangle)>0$. The groups satisfying both conditions will be called {\it Klein subgroups}.

We start by excluding the case when the genus of the curve is even, using the following fact.
\begin{lem}
Let $H$ be a hyperelliptic curve of genus $g(H)=2k$. Then, there does not exist a Klein subgroup $\langle \s,\t \rangle \subset\Aut(H)$. 
\end{lem}
\begin{proof}
 If $g(H)$ is even, then $|W|=4k+2$, hence the action of the subgroup $\langle \s,\t\rangle \cong\ZZ_2\times\ZZ_2$ cannot be free on $W$. On the other hand, Proposition \ref{prop:Angela} and Figure \ref{fig:hyperelliptic_branched+2} show that the ramification points of any double covering cannot be Weierstrass, see also \cite[Lemma 1]{GS}. 
\end{proof}

Now, we are left with two cases, either $g(H)=4k+1$ or $g(H)=4k+3$. The next proposition is essentially rephrasing \cite[Lemma 2.13]{BOklein}.
\begin{prop}\label{prop:cominv}
Let $H$ be a hyperelliptic curve with the group of commuting involutions $\langle \iota,\s,\t\rangle \subset\Aut(H)$. Then
\begin{itemize}
    \item there exists a unique Klein subgroup of fixed point free involutions if and only if $g(H)=4k+1$. 
    \item there exists a unique Klein subgroup of involutions with fixed points if and only if $g(H)=4k+3$. 
\end{itemize}
\end{prop}
\begin{proof}
Assume $g(H)=2n+1$.
Without loss of generality, by Proposition \ref{prop:involutions}, we can assume $\s,\t$ are fixed point free. Then the existence of the group of the fixed point free involutions is equivalent to the fact that the involution $\s\t$ is fixed point free. 

Denote by $g_{\alpha}$ the genus of the quotient curve $H/\alpha$ for $\alpha$ an involution and by $g_0$ the genus of $H/\langle\s,\t\rangle$. According to Accola's Theorem (\cite[Theorem 5.9]{A}) for the group $\langle \s,\t\rangle$ we obtain 
\begin{eqnarray*}
2(2n+1) + 4g_0 & = & 2g_{\sigma} + 2g_{\tau} + 2g_{\sigma\tau}\\
2n+1 +2g_0 & = &g_{\sigma} + g_{\tau}  + g_{\sigma\tau} \\
 &=& 2n + 2 + g_{\sigma\tau}.
\end{eqnarray*}
Since the left-hand side is odd, $\s\t$ is fixed point free if and only if $n=2k$.

Analogously, the group $\langle\iota\s,\ita\rangle$ contains $\s\t$, so it contains only involutions with fixed points if and 
only if $n=2k+1$.

The uniqueness of the groups follows from the fact that any other subgroup contains $\iota$ or contains both fixed-point free involutions and involutions with fixed points.
\end{proof}

\section{Hyperelliptic double coverings}\label{sec:higher}
In this section, we focus on the bottom-up perspective.
According to Proposition \ref{prop:involutions},  there are three possibilities for a hyperelliptic double covering, namely \'etale coverings and  
coverings branched in 2 or 4 points.

\subsection{Coverings branched in 2 points}
Let us assume $H$ is hyperelliptic and $f: C\ra H$ is a covering branched in 2 points. Firstly, we  show a necessary and sufficient condition for $C$ to be hyperelliptic.

\begin{prop}\label{prop:Angela}
Let $f: C\ra H$ be a covering of a hyperelliptic curve $H$ branched in 2 points $P,Q\in H$. Then $C$ is 
hyperelliptic if and only if $P=\iota Q$ and the line bundle defining the covering is $\cO_H(w)$ for some 
Weierstrass  point $w$. 
\end{prop}
\begin{proof}
Let $\eta \in \Pic^1(H)$ be the element defining the covering $f: C\ra H$, so $\eta^2 = \cO_H(P + Q)$. 
Suppose $P= \iota Q$ and $\eta = \cO_H(w)$ with $w$ a Weierstrass point. By the projection formula

\begin{eqnarray*}
h^0(C, f^*\eta ) &= & h^0(H, \eta \otimes f_*\cO_C) \\
& = & h^0(H, \eta ) +  h^0(H, \cO_H) \\
& =&  2 
\end{eqnarray*}
Since $\deg f^*\eta =2$, this implies that $C$ is hyperelliptic.

Now assume that $C$ is hyperelliptic. Let $h_C$ and $h_H$ be the hyperelliptic divisors on $C$, respectively on $H$. Notice that the hyperelliptic involution on $C$ is a lift of the hyperelliptic
involution on $H$. Since every automorphism of $C$ commutes with the hyperelliptic involution, the ramification 
locus of $f$ is invariant under the hyperelliptic involution, so either it consists of two points conjugated to 
each other or of two Weierstrass points. In the latter case, $\eta$ is a square root of $\cO_H(w_1 + w_2)$,
where $w_1, w_2$ are Weierstrass points, but a necessary condition for the hyperelliptic involution $\iota$ to lift 
to an involution on $C$ is $\iota^* \eta \simeq \cO_H(h_H) \otimes \eta^{-1} \simeq \eta$, that is,  $\eta^2 \simeq
\cO_H(h_H)$, a contradiction. Therefore, the branch locus is of the form $\{ P, \iota P \}$. By the projection 
formula
$$
2= h^0(C, \cO_C(h_C) ) =  h^0(H,  f_*(f^*\cO_H(w))) = h^0(H, \cO_H(w)) +  h^0(H, \cO_H(w) \otimes \eta^{-1})
$$
with $w \in H$ a Weierstrass point. This implies that $\eta$ is of the form $\cO_H(w)$. 
\end{proof}
One constructs a commutative diagram of hyperelliptic curves 
(left Diagram \ref{diag:Const_tildeC2})
starting from $2g+3$ given points in $\PP^1$. Let $[y],[z],[w_1],\ldots,[w_{2g+1}]\in\PP^1$.
Let $H$ be the  hyperelliptic genus $g$ curve ramified in $z,w_1,\ldots,w_{2g+1}$ mapping to the corresponding points with brackets in $\PP^1$,  and let $y_1,y_2$ be the
fibre over $[y]$. Let $f:C\ra H$ be a double covering branched in $y_1,y_2$ and defined by $\cO(z)$. The hyperelliptic curve $C$ of genus $2g$ can also be 
constructed in the following way. Let $p:\PP^1\ra\PP^1$ be the double covering branched in $[y],[z]$. For $i=1,2$ 
denote by $[y'],[z'],[w_1^{i}],\ldots,[w_{2g+1}^i] \in\PP^1$ the respective preimages. Then $C$ is a double covering of $\PP^1$
branched in $[w_1^{i}],\ldots,[w_{2g+1}^i]$. Clearly, the preimages of $[y'],[z']$ in $C$ coincide with the appropriate preimages of $y_1,y_2,z$ (see right Diagram \eqref{diag:Const_tildeC2}). 

\begin{equation} \label{diag:Const_tildeC2}
\xymatrix@R=.7cm@C=.7cm{
& C_{2g} \ar[dr]^{2:1} \ar[dl]_{2:1} & \\
H' _{g}\ar[dr]_{2:1} & & H_g \ar[dl]^{2:1}\\
 &\PP^1& } \hspace{2.6cm}
 \xymatrix@R=1.1cm@C=1.1cm{
 C \ar[d]_{2:1} \ar[r]^{f} & H \ar[d]^{2:1} \\
\PP^1 \ar[r]^p & \PP^1 
}
\end{equation}
\begin{figure}[h!]
\includegraphics[width=13cm,  trim={0.3cm 9cm 0cm 8cm}, clip]{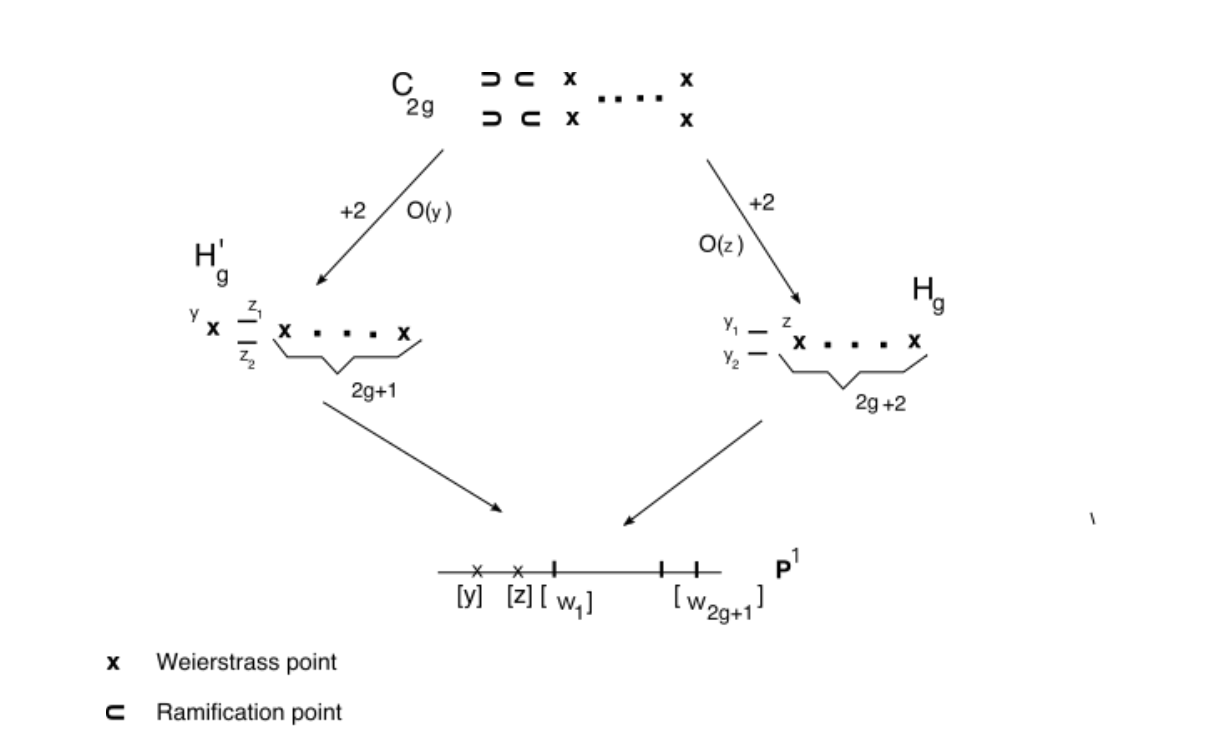}
  \caption{Hyperelliptic coverings ramified in two points}
 \label{fig:hyperelliptic_branched+2} 
\end{figure}

According to \cite{M}, the the Prym variety of 
an étale double covering $f:C\ra H$ over an hyperelliptic curve is isomorphic to the product of the Jacobians of the curves, obtained as the quotients by the two other involutions on $C$. 
In this case, since $C$  is hyperelliptic one of these quotient curves is isomorphic to $\PP^1$ and the  other is the hyperelliptic curve $H'$ of genus g, appearing in the left Diagram \ref{diag:Const_tildeC2}, defined by exchanging the role of $y$ and $z$. Hence, the Prym variety $P(C/H)$ of the covering $f:C\ra H$ is isomorphic to $JH'$.
The distribution of the Weierstrass points is illustrated in Figure \ref{fig:hyperelliptic_branched+2}.
\begin{cor}\label{never2}
The construction shows that the Prym map of a double covering branched in 2 points is never injective, since the Jacobian of $H'$ does not recognise the branching points. In particular, if one moves a point $[z]$ in $\PP^1$, one gets a one dimensional family of coverings $C\ra H$ with the same Prym.
\end{cor}

\subsection{\'Etale coverings and coverings branched in 4 points}
Now, we consider a hyperelliptic curve $Y$ of genus $g\geq 2$ and a double \'etale covering $\pi: X \ra Y$, so $X$ is of genus 
$2g-1$, that is, $Y$ is the quotient of $X$ by a fixed point free involution $\tau$.  
According to 
\cite[Proposition 4.2]{bo} $X$ is hyperelliptic if and only if the 2-torsion point $\eta$ defining $\pi$ is of the form $\eta=\cO_Y(w_1-w_2)$
where $w_1$ and $w_2$ are Weierstrass points.  Assume that $X$ is hyperelliptic and let $\iota$ be the hyperelliptic involution. Let 
$Y':= X / \langle \iota\tau \rangle$, which is   of genus $g-1$.  The double covering $f: X \ra Y'$ is ramified in four points. The
hyperelliptic involution $\iota$ on $X$ descends to an hyperelliptic involution on $Y'$, denoted by $j$ (see Diagram \eqref{diag:diag3}).

Conversely, starting from a hyperelliptic curve $Y'$ of genus $g-1$ we can give a necessary and sufficient condition 
for $X$ to  be hyperelliptic.
\begin{prop} \label{hyp-cov}
Let $Y'$  be a hyperelliptic curve of genus $g-1$ and $f: X \ra Y'$ a double covering ramified in
four points defined by a line bundle $\eta \in \Pic^{2}(Y')$, such that  $\eta^2 \simeq \cO_{Y'}
(B)$, where $B$ is the branch
locus of the covering. 
Then $X$ is hyperelliptic if and only if $\eta = \cO_{Y'}(h_{Y'}) $, with $h_{Y'}$ the hyperelliptic divisor on $Y'$
and $B \in |2h_{Y'}|$ is reduced.

\end{prop}

\begin{proof}
Since $f_*\cO_X \simeq  \cO_{Y'} \oplus \eta^{-1}$, by using the projection formula  one computes
$$
H^0(X,   f^* \cO_{Y'}(h_{Y'}) ) =  H^0(Y',  \cO_{Y'}(h_{Y'}) )  \oplus H^0(Y',  \cO_{Y'}). 
$$
So $\dim H^0(X,   f^* \cO_{Y'}(h_{Y'}) )  = 3$.  According to Clifford's Theorem \cite[Chapter III]{acgh} $X$ is
hyperelliptic and 
$ f^* h_{Y'}$ is a multiple of the hyperelliptic divisor.  Suppose that $X$ is hyperelliptic and the double covering $f:X \ra Y'$ is given by a line bundle $\eta $ such that $\eta^2 \simeq \cO_{Y'}(B)$, with $B$ the 
(reduced) branch locus of $f$.  From the 
commutativity of the Diagram \eqref{diag:diag3} the union of $B$ and the set of Weierstrass points of $Y'$ map to the
branch locus of the map $Y \ra \PP^1$, which has cardinality $2g+2$. This implies that $B \in |2h_{Y'}|$, so $\eta $
is a square root of $\cO_{Y'}(2h_{Y'})$. Since 
$X= \Spec(\cO_{Y'} \oplus \eta^{-1})$ the involution $j$ on $Y'$ lifts to an involution on $X$ if and only if 
$j^* \eta \simeq \eta$. Then, either $\eta = \cO_{Y'}(h_{Y'})$ or $h^0(Y', \eta ) =1$. In the latter case, if 
$\eta=\cO_{Y'}(p_1+q_1)$, then $j(p_1) = p_1 $ and $j(q_1) = q_1$, that is, $\eta$ is defined by the sum of 
Weierstrass points, say $\eta= \cO_{Y'}(w_1+w_2)$. Since $X$ is hyperelliptic, $f^*\eta \in |2h_X|$ but this contradicts the projection formula. 
Therefore, $\eta = \cO_{Y'}(h_{Y'}) $. 

\begin{equation} \label{diag:diag3}
\xymatrix@R=.7cm@C=.7cm{
& X_{2g-1} \ar[dr]^{\pi} \ar[dl]_{f} & \\
Y' _{g-1}\ar[dr]_{2:1} & & Y_g \ar[dl]^{2:1}\\
 &\PP^1& }\hspace{2.6cm}
 \xymatrix@R=1.1cm@C=1.1cm{
 X_{2g-1} \ar[d]_{2:1} \ar[r]^{f} & Y'_{g-1} \ar[d]^{2:1} \\
\PP^1 \ar[r]^p & \PP^1 
}
\end{equation}

One can see this construction from the perspective of points in $\PP^1$. Let $[x],[y],[w_1],\ldots,[w_{2g}]\in\PP^1$ and
let $Y_g$ be the  hyperelliptic genus $g$ curve branched in these points. Let $X_{2g-1}\ra Y_g$ be the \'etale double covering defined by $\cO(x-y)$,
where $ x, y$ are the preimages of 
$[x], [y]$ respectively. On the other hand  $X_{2g-1}$ can be also
constructed in the following way. Let $p:\PP^1\ra\PP^1$ be the double covering branched in $[x],[y]$. For $i=1,2$ 
denote by $[w_1^{i}],\ldots,[w_{2g}^i]\in \PP^1$ the respective preimages under $p$. Then $X_{2g-1}$ is a double covering of $\PP^1$
branched in $[w_1^{i}],\ldots,[w_{2g}^i]$. 
The curve $Y'_{g-1}$ is constructed as double cover of $\PP^1$ branched in $[w_1],\ldots,[w_{2g}]$ and one obtains the commutativity of the right Diagram \eqref{diag:diag3}. The covering $X_{2g-1}\ra Y'_{g-1}$ is branched in 
$x,\iota x,y,\iota y$ and defined by the hyperelliptic bundle (see Figure 
\ref{fig:double_hyperelliptic_covers}). 

\begin{figure}[h!]
\includegraphics[width=11.3cm,  trim={0.3cm 8.5cm 0cm 8.2cm}, clip]{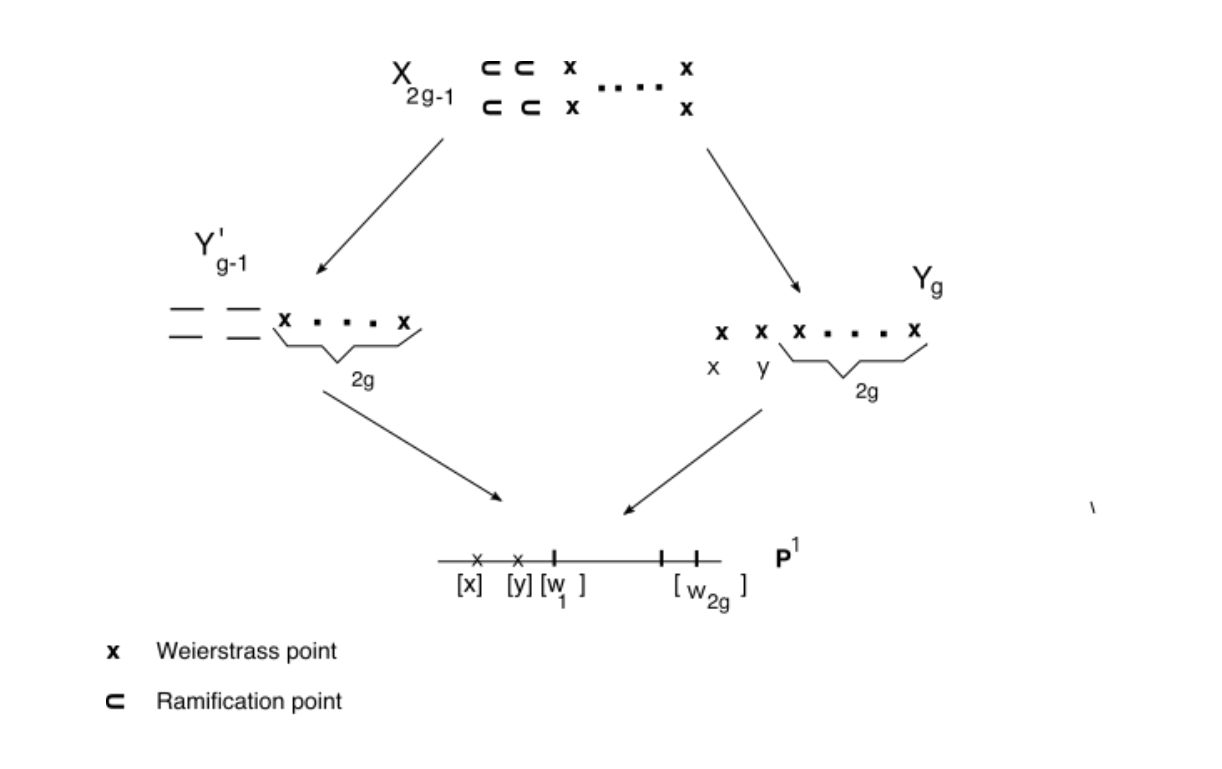}
\caption{Distribution of Weierstrass points on hyperelliptic covers}
\label{fig:double_hyperelliptic_covers} 
\end{figure}

\bigskip

\end{proof}

\subsubsection{The Prym map}
Let $\cR_{g-1,4}^H:=\{(Y' , B) \ \mid \  B \in |2h_{Y'}|\}$ be the space parametrising double coverings $f:X \ra Y'$
ramified in four points where both curves are hyperelliptic, according to the previous proposition this depends only 
on the choice of the branch divisor in $|2h_{Y'}|$. We denote by $\cJ_g^H \subset \cA_g$ the locus of the 
hyperelliptic Jacobians inside  of the moduli space of principally polarised abelian varieties of dimension $g$, and by $\cJ_g^{H, (1,2, \dots, 2)}$  the moduli of abelian varieties which are quotients of hyperelliptic Jacobians by 2-torsions of the form $w_i-w_j$.
Let $\cR_{g}^H:= \cR_{g,0}^H$ be the moduli space of hyperelliptic étale double coverings over curves of genus $g$. For $b=0,4$ we define  the Prym map $\Pr_{g,b}$ as the map which associates to a hyperelliptic double covering $[X\ra Y] \in \cR_{g,b}^H$  its Prym variety $P(X/Y)$.

\begin{prop}\label{twocases}
The relation given by left Diagram \eqref{diag:diag3} induces an isomorphism 
$$
\gamma: [f:X \ra Y'] \mapsto [\pi: X \ra Y]
$$ 
fitting in the  following commutative diagram

\begin{equation} \label{diag4}
\xymatrix@R=1cm@C=1cm{
\cR_{g-1,4}^H \ar[rr]^{\gamma} \ar[dd]_{\Pr_{g-1,4}}  \ar@{->}[rdrd]& & \cR_{g}^H \ar[dd]^{\Pr_{g,0}} \ar@{->} [ldld] \\
&  {} & \\
\cJ_g^{H, (1,2, \dots, 2)}   & & \cJ_{g-1}^H  
}
\end{equation}
 where the diagonal arrows are the corresponding forgetful maps. In particular $\deg
 \Pr_{g-1}\mid_{\cR_{g-1,4}^H} = {2g+2 \choose 2}$
 and $\Pr_{g}\mid_{\cR_{g}^H}$ is a $\PP^2$-bundle. 
\end{prop}

\begin{proof}
It is well-known (\cite{M})  that the Prym variety of an \'etale double covering $X \ra Y$ over a hyperelliptic curve is 
isomorphic to the product of the Jacobians of the quotient curves $X/\iota \simeq \PP^1$ and $X/\iota \tau$, which gives $P(X / Y) \simeq JY'$ as a principally polarized abelian variety.  
This proves the commutativity of the top right triangle of the diagram. Similarly, we have
$P(X / Y') \simeq JY/ \langle  \beta \rangle$, with $\beta=w_i-w_j \in JY[2]$ the element defining the \'etale 
covering $\pi$. This shows the commutativity of the top left triangle of the diagram.
\end{proof}

\begin{cor}\label{never}
Hyperelliptic Prym maps of  \'etale double coverings and double coverings branched in 4 points are 
never injective. 
\end{cor}

\subsection{Useful notation}
We recall the following notation from \cite{bo} that helps with dealing with abelian subvarieties. Let X be an abelian variety and $M_i$ abelian varieties such that there
exist embeddings $M_i \hookrightarrow X$ for $i = 1,...,k$. We write
$$
 X = M_1 \boxplus M_2  \ldots \boxplus M_k
$$
  if $\epsilon_{M_1} + \epsilon_{M_2} + \ldots +\epsilon_{M_k} = 1$, where $\epsilon_{M_i}$ are the associated symmetric idempotents. In particular, 
$X = M \boxplus N $ if and only if $(M,N)$ is a pair of complementary abelian subvarieties of $X$. If $M_i$'s are general enough, then
the decomposition is unique up to permutation, see
\cite[Proposition 5.2]{bo}.
We will also use the following notation. If $f:X\to Y$ is a covering and $f^*$ is not an embedding we will denote the image $\im(f^*(JY))$ by $JY^*$.

In the sequel we will denote by the same letter an automorphism of the covering curve and its extension to the Jacobian, except for the hyperelliptic involution, whose extension is $-1$. We will also denote the identity as $1$. By $m_k$ we denote the multiplicity by $k$ on an abelian variety.

\bigskip

\section{Prym maps of hyperelliptic \'etale Klein coverings}\label{etaleKlein}
In \cite{BOklein}, we have considered \'etale Klein coverings of genus 2 curves and have shown that the Prym map is injective in this case. We now generalise the result to hyperelliptic Klein coverings of higher genera. Recall that a Klein subgroup $\langle \eta,\xi \rangle \simeq \ZZ_2 \times \ZZ_2$ of $JH[2]$ is called isotropic, respectively non-isotropic, if it is isotropic (resp.  non-isotropic) with respect to the Weil form $e: JH[2]\times JH[2]\ra \FF_2$.

\subsection{From the bottom construction}
Let  $H$ be a genus $g$ hyperelliptic curve with 
Weierstrass points $w_1,\ldots,w_{2g-1},x,y,z\in H$ 
and let $[w_1],\ldots,[w_{2g-1}],[x],[y],[z]\in \PP^1$ be the corresponding set of $2g+2$ branched points. Set $\eta=\cO_H(x-y),\  \xi=\cO_H(y-z)$.
According to 
\cite[Theorem 4.7]{bo}, the covering $\tC$ associated to the non-isotropic Klein group  $G=\{0,\eta,\xi,\eta+\xi\} \subset JH[2]$ is hyperelliptic and since 
the $4:1$ map $\tC\ra H$ is \'etale, $\tC$ is of genus $4g-3$.
Let $C_x$ be the double covering of $H$ defined by $\x$, $C_y$ defined by $\e+\x$ and $C_z$ defined by $\e$; all three of genus $2g-1$ Then the Prym varieties of these coverings are Jacobians of curves, denoted by $H_x,\ H_y,\ H_z$ respectively, of genus $g-1$. Recall that for $j\in\{x,y,z\}$, the curve $H_j$ is given by choosing $[j], [w_1],\ldots,[w_{2g-1}]$ as branch points.
These curves fit into the following commutative diagram (we draw only two curves to make it easier to read).
\begin{equation} \label{diagetale}
\xymatrix@R=.6cm@C=.6cm{
&& \tC \ar[dr]^{et} \ar[dl]_{et} &&  \\
& C_{x}\ar[dr]^{et} \ar[dl]_{+4} & & C_{y}\ar[dr]^{+4} \ar[dl]_{et} & \\
H_{x}\ar[drr] &&H \ar[d] && H_{y}\ar[dll] \\
&& \PP^1 &&
}
\end{equation}
Here $+4$, denotes a 2:1 map branched in 4 points and 
{\it et} stands for an étale 2:1 map. 
\subsection{Decomposition of $J\tC$}
 In order to decompose the Jacobian of $\tC$ and describe the Prym variety $P(\tC/H)$
of the covering $\tC \ra H$,
we will use a top-down perspective. Let $\tC$ be a hyperelliptic curve of genus $4g-3$ with commuting fixed point free involutions $\s,\t,\s\t\in\Aut(\tC)$. Without loss of generality, we can assume $C_x=\tC/\s, C_y=\tC/\t, C_z=\tC/\s\t$ and $H=\tC/\langle \s,\t\rangle $. With this notation, we have that $H_x=\tC/ \langle \s,\iota\t \rangle$ and the following diagram commutes
\begin{equation}\label{diag:injectionJHx}
\xymatrix@R=.9cm@C=.6cm{
& \tC \ar[dr]^{+4} \ar[dl]_{et} \ar[d]_{+4} & \\
C_{\s}\ar[dr]_{+4} &  C_{\iota \t} \ar[d]_{+2} & C_{\iota\s\t} \ar[dl]^{+2}\\
&H_x& 
}
\end{equation}
where $C_{\alpha}=\tC/\alpha$ with $\alpha$ an involution. Analogously one checks that $H_y=\tC/\langle \t,\iota\s\rangle$ and $H_z=\tC/\langle \s\t,\iota\t\rangle$.
\begin{prop} \label{DecompJC-et}
The Jacobian of $\tC$ is decomposed in the following way
$$J\tC=JH^*\boxplus JH_x\boxplus JH_y\boxplus JH_z.$$
In particular, $P(\tC/H)=JH_x\boxplus JH_y\boxplus JH_z$.
\end{prop}
\begin{proof}
The proof follows from straightforward computation. Firstly, note that Diagram \ref{diag:injectionJHx} shows that $JH_x$ is embedded in $JC_{\iota\s\t}$ which is embedded in $J\tC$, hence $JH_x$ is embedded in $J\tC$ with the restricted polarisation type being four times the principal polarisation on $JH_x$. Analogously, $JH_y$ and $JH_z$ are also embedded in $J\tC$.

Since the covering $\tC\ra H$ is \'etale, by \cite[Proposition 11.4.3]{bl}, the pullback map is not an embedding, so we denote the image of $JH$ as $JH^*$. Moreover, $JH^*=\im(1+\s+\t+\s\t)$. Since the hyperelliptic involution extends to $(-1)$ on $J\tC$ we have that $JH_x=\im(1+\s-\t-\s\t)$, $JH_y=\im(1-\s+\t-\s\t)$,
$JH_z=\im(1-\s-\t+\s\t)$. Observe that sum of the endomorphisms defining these Jacobians is 4,  which is also the exponent of each subvariety in 
$J\tC$, so $\epsilon_{JH^*}+\epsilon_{JH_x}+\epsilon_{JH_y}+\epsilon_{JH_z}=1$.
Since, by definition, the Prym variety is complementary to $JH^*$,  we get that $P(\tC/H)=JH_x\boxplus JH_y\boxplus JH_z$.
\end{proof}

\begin{prop}\label{4.2}
The addition map $\psi:JH_x\times JH_y\times JH_z\lra P(\tC/H)$ is a polarised isogeny of degree $4^{2g-2}$ and its kernel is contained in the set of 2-torsion points.
\end{prop}
\begin{proof}
In order to compute 
the kernel of $\psi$ we consider the description of $JH_j$, for $j\in\{x,y,z\}$, as fixed loci inside the 
Jacobian of $\tC$:
$$
 JH_{x}  \subset  \Fix(\sigma,  \iota\tau),  \quad JH_{y}\subset \Fix(\tau,  \iota\sigma),  \quad JH_{z}\subset\Fix(\s\tau, \iota\sigma).
$$ 
Let $(a,b,c) \in \Ker \psi $, then $c=-a-b$. Applying $\iota\sigma$  and  $\iota\tau$ to $c=-a-b$ we get 
$$
-a-b=\iota\sigma (-a-b) = -\iota a -b   \quad \textnormal{ and } \quad -a-b=\iota\tau (-a-b) = -a - \iota b  ,
$$
so $\iota a=a$ and $\iota b= b$,  that is, $a,b$ and $c$ are 2-torsion points in their respective Jacobians. This implies
\begin{equation}\label{eq:kervarphi}
    \Ker \psi = \{ (a,b, -a-b) \ \mid \  a \in JH_{x}[2], \ b \in JH_{y}[2] \}.
\end{equation}
The restricted polarisation to $JH_j$ is of type $(4,\ldots,4)$. Since $P(\tC/H)$ is complementary to $JH^*$, it has complementary type which is $(1,\ldots,1,1,4,\ldots,4)$ with $g-1$ fours.
Moreover, $\psi$ as an addition map is polarised and the degree is exactly $4^{2(g-1)}$.
\end{proof}
\begin{rem}
    We included a proof of Proposition \ref{4.2} for the sake of completeness, although it is proven in  \cite{RR} (unpublished) and in the recently published book \cite[Corollary 5.2.6]{LR}.
\end{rem}

\subsection{The Prym map}
 Let $\mathcal{RH}_{g,0}$ denote the moduli space parametrising the pairs $(H, G)$ with $H$ a hyperelliptic curve of genus $g$ and $G$ a Klein subgroup of $JH[2]$ whose generators are differences of Weierstrass points, so the corresponding covering curve $\tC$ is hyperelliptic of genus $4g-3$. We call the elements of $\mathcal{RH}_{g,0}$ 
\'etale hyperelliptic Klein coverings. Set
$$
\delta:= (\underbrace{1,\ldots,1}_{2g-2}, \underbrace{4, \ldots, 4}_{g-1})$$
and let $\cA_{3g-3}^{\delta}$ denote the moduli space of
polarised abelian varieties of dimension $3g-3$
 and polarisation of type $\delta$. The Prym map 
 associates to the hyperelliptic Klein covering 
 $\tC \ra H$ induced by $(H, G)$, the polarised Prym variety $(P(\tC/H), \Xi)$, where $\Xi$ is the 
 restriction to $P(\tC/H)$ of the principal polarisation on $J\tC$.
 
The main aim of this section is to prove that the Prym map 
$$
 Pr_{4g-3,g}^H: \mathcal{RH}_{g,0} \ra \cA_{3g-3}^{\delta}, \qquad (H,G) \mapsto (P(\tC/H), \Xi)
$$
of \'etale hyperelliptic Klein coverings is injective. We will show this by constructing the inverse map explicitly. We start by showing 
the following equivalence of data, which generalises 
\cite[Theorem 3.1]{BOklein}.

\begin{prop} \label{equivalence1}
The following data are equivalent:
\begin{enumerate}
    \item a triple $(H, \eta, \xi)$, with $H$ a hyperelliptic curve of genus $g$ and $\eta$ and $\xi$ differences of 
    Weierstrasss points such that Klein subgroup $G=\langle \eta, \xi \rangle$ of $JH[2]$ is non-isotropic; 
    \item a hyperelliptic curve $\tC$ of genus $4g-3$ with 
    $\ZZ_2^3 \subset \Aut(\tC)$;
    \item a hyperelliptic curve $H$ of  genus $ g$ together with the choice of $3$  Weierstrass points; 
    \item a set of $2g+2$ points in $\PP^1$ with a chosen triple of them, up to projective equivalence (respecting the triple).
\end{enumerate}
\end{prop}

\begin{proof}
Equivalences $(1) \Leftrightarrow (3)  \Leftrightarrow (4)$ are obvious. The equivalence $(3) \Leftrightarrow (2) $ follows from \S 3.1.
\end{proof}
\begin{cor}
The moduli space $\mathcal{RH}_{g,0}$ is irreducible.
\end{cor}
\begin{proof}
It follows from the equivalence $(1) \Leftrightarrow (4)$ of 
the Proposition \ref{equivalence1}.
\end{proof}
 
\begin{lem} \label{multtwo}
Let $JH_x,JH_y,JH_z$ be as before and let $P=P(\tC/H)$. Let $Z=(JH_x \times JH_y \times JH_z)[2]$ be the set of 2-torsion points on the product and let $G_P=\psi(Z)$.
By $m_2$ we denote the multiplication by 2.
Consider the following commutative diagram
\begin{equation} 
\xymatrix@R=.9cm@C=.9cm{
JH_x \times JH_y \times JH_z   \ar[r]^(.7){\psi} \ar[d]_{m_2} & P 
\ar[d]^{\pi_{P}} \\
JH_x \times JH_y \times JH_z    \ar[r]^-{p}&  P/G_P   
}
\end{equation}
Then the map $p$ is a polarised isomorphism of principally polarised abelian varieties.
\end{lem}
\begin{proof}
Note that $JH_x \times JH_y \times JH_z$ in the top left has product polarisation of type four times the principal one. Hence, $Z$ is an isotropic subgroup of the kernel of the polarising map. In particular $m_2$, having $Z$ as its kernel, is a polarised isogeny (see also \cite[Cor. 2.3.6]{bl}). Moreover, $Z$ is also the kernel of $\pi_P\circ\psi$, hence both $\psi$ and $\pi_P$ are polarised isogenies. Then, the isomorphism theorems yield the existence and uniqueness of the isomorphism $p$.
\end{proof}
\begin{cor}\label{3.5}
Let $\Xi$ be the restricted polarisation on $P$ and $\phi_\Xi$ its polarising isogeny. Then $G_P=\ker(\phi_\Xi)\cap P[2]\simeq\ZZ_2^{2g-2}$.
\end{cor}
\begin{proof}
Clearly $G_P \subset P[2]$ and since $P/G_P$ is principally polarised, $G_P$
is a maximal isotropic subgroup of $\ker(\phi_\Xi)$.
Hence $G_P\subset \ker(\phi_\Xi)\cap P[2]$ and the 
claim follows by computing the cardinalities.

\end{proof}

\begin{lem} \label{Intersection}
It holds $G_P = JH_x \cap  JH_y \cap JH_z$.
\end{lem}
\begin{proof}
Let $a \in JH_x \cap  JH_y \cap JH_z$,  then $a$ is fixed by all the involutions in $J\tC$, in particular it is a 2-torsion point. Hence $a=a+a+a
\in \im \psi(Z)$. Conversely, if $a+b+c \in G_P$, with $a \in JH_x=\Fix(\s,\t)^0$, $b \in JH_y=\Fix(\iota\s,\t)^0$  and $c \in JH_z= \Fix(\iota\s,\s\t)^0$, then $a,b,c \in P[2]$. One checks that 
$a,b,c \in JH_x \cap  JH_y \cap JH_z$. 

\end{proof}
Let $\pi_j : \tC \ra H_j$ be the 4:1 branched maps in the diagram (4.1)
for $j\in \{x,y,z\}$ an let $k: \Pic^4(\tC) \ra\Pic^0(\tC) $ defined
by $D \mapsto D-2g_2^1$, with $2g_2^1$ the hyperelliptic divisor. Hence, we have injective maps
$$
\alpha_j := k\circ \pi_j^* : \Pic^1(H_j) \ra\Pic^0(\tC) \simeq J\tC.
$$

\begin{prop}\label{gluing}
The maps $\alpha_j$ have as common image in $G_P \subset P[2]$  for 
$j\in \{x,y,z\}$, the image of $2g-1$ Weierstrass points on each of the curves $H_j$.
\end{prop}
\begin{proof}
Observe that by construction
$$
\alpha_x(\Pic^1 H_x) \cap \alpha_y(\Pic^1 H_y) \cap \alpha_z(\Pic^1 H_z) 
\subset JH_x \cap  JH_y \cap JH_z =G_P.
$$
On the other hand, for any Weierstrass point $w \in H_x$, image of 
a Weierstrass point $\tilde w \in \tC$ we have  
$$\pi_x^* w \sim \tilde w + \s\tilde w+ \iota\t\tilde w+ \s\t\tilde w
 \sim \tilde w + \s\tilde w+ \t\tilde w+ \s\t\tilde w $$
since $\iota\tilde w = \tilde w$. Similarly, for every $w= \pi_y(\tilde w)$ or $w= \pi_z(\tilde w)$  image of a 
Weierstrass point in $\tC$, one checks that 
$$\pi_j^* w \sim \tilde w + \s\tilde w+ \t\tilde w+ \s\t\tilde w.$$
\end{proof}

\begin{thm}\label{injectiveKlein}
For $g\geq 2$ the hyperelliptic Klein Prym map $Pr_{4g-3,g}^H: 
 \cRH_{g,0} \ra \cA_{3g-3}^{\delta} $ is injective. 
\end{thm}
\begin{proof}
Let $(P,\Xi)$ be a polarised variety in the image of the Prym map. Let $G=P[2]\cap \ker\phi_\Xi$, which is an isotropic subgroup in $\ker\phi_\Xi$. Let $\pi_P:P\ra P/G$. According to Corollary \ref{3.5}, we have that $P/G$ is principally polarised and by Lemma \ref{multtwo}, $P/G$ is polarised isomorphic to a product of the form  $JH_1\times JH_2\times JH_3$,
for uniquely determined Jacobians $JH_1$, $JH_2$ and $JH_3$.

Now, since $P$ is in the image of the Prym map and the $H_i$, $i\in\{1,2,3\}$,  are actually isomorphic to $H_j$, for $j\in\{x,y,z\}$ 
corresponding to some points $[x],[y],[z] \in\PP^1$. Let $f_i:H_i\ra \PP^1$ be the hyperelliptic maps and $W^i$ the set of Weierstrass points
on each $H_i$. By Proposition \ref{gluing} there exists an automorphism of $\PP^1$ such that $\bigcap f_i(W^i)$ consists of $2g-1$ points in 
such a
way that $\alpha_1(w_n^1)=\alpha_2(w_n^2)=\alpha_3(w_n^3) \in G_P$, for all $n\in \{1, \ldots, 2g-1\}$.
Moreover, there are  uniquely determined points $[x],[y],[z]\in \PP^1$ that are images of the remaining Weierstrass points. In this way, we have constructed the set $\{[w_1],\ldots,[w_{2g-1}], [x],[y],[z]\}$ of $2g+2$ points in $\PP^1$ with a distinguished triple.
Hence, the obtained map $(P, \Xi)\mapsto \{[w_1],\ldots, [w_{2g-1}], [x],[y],[z] \}$ provides the inverse to the Prym map via the equivalence in Proposition \ref{equivalence1}. 
\end{proof}
\begin{rem}
Note that for $g=2$, the curves $H_i$ are actually elliptic, so one uses the notion of 2-torsion points instead of Weierstrass points and some steps are vacuous, see \cite{BOklein}.
\end{rem}

\begin{figure}[h!]
\includegraphics[width=11.3cm,  trim={0.3cm 6cm 0cm 7cm}, clip]{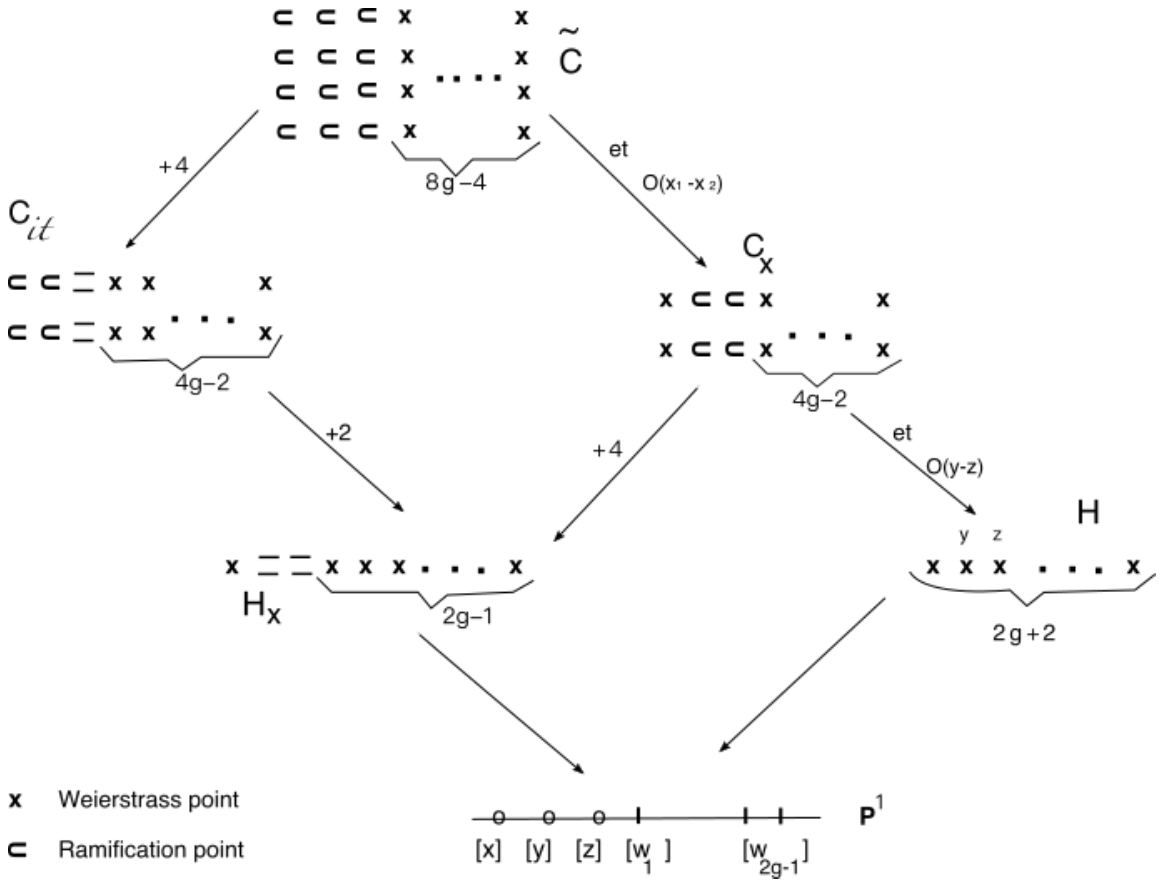}
  \caption{Weierstrass points on hyperelliptic Klein coverings}
 \label{fig:hyperelliptic_covers} 
\end{figure}

\subsection{The mixed case}

Consider a smooth curve $\tC$ of genus $4g-3$, with $g\geq 2$,
admitting  one  fixed point free involution $\s$ and two 
involutions $\iota\t$, $\iota\s\t$ with 4 fixed points each. 
The corresponding tower of curves is the same as in the case
of \'etale Klein coverings (see Diagram \eqref{diagetale} or Figure \ref{fig:hyperelliptic_covers}). 
The starting data is now the base curve $H_{\s, \iota\t}$ 
(for instance  $H_x$ in the diagram) of genus $g-1$, a 
choice of two pairs of conjugated points (branch points for the double covering $C_x \ra H_x$) and a Weierstrass point. According to
Proposition \ref{DecompJC-et},  $ J\tC = JH^* \boxplus JH_x \boxplus JH_y \boxplus JH_z $, therefore the associated Prym variety
is $P= P(\tC / H_x)= JH^*\boxplus JH_y \boxplus JH_z$. Moreover, since the restricted polarisation to $JH_x$ is of type $(4, \ldots, 4)$, the Prym variety $P$  is of type
$$
\delta:= (\underbrace{1,\ldots,1}_{2g-1}, \underbrace{4, \ldots, 4}_{g-1}).
$$
Consider the canonical addition map
$$
\psi: JH^* \times JH_y \times JH_z \ra P.
$$
Since $JH^* \subset \Fix(\s, \t)$, $JH_y \subset \Fix(\t, \iota \s)$ and $JH_z \subset \Fix(\s\t, \iota\s)$, we can conclude as before that 
\begin{equation}
    \Ker \psi = \{ (-(b+c),b, c) \ \mid \  b \in JH_{y}[2], \ c \in JH_{z}[2] \}.
\end{equation}
Analogously to Lemma \ref{multtwo} we have:

\begin{lem} \label{multtwo2}
Let $JH,JH_y,JH_z$ be as before and let $P=P(\tC/H_x)$. Let $Z=(JH \times JH_y \times JH_z)[2]$ be the set of 2-torsion points on the product and let $G_P=\psi(Z)$.
By $m_2$ we denote the multiplication by 2.
Then one obtains the following commutative diagram
\begin{equation} 
\xymatrix@R=1.2cm@C=.9cm{
JH \times JH_y \times JH_z   \ar[rr]^(.5){\psi'} \ar[dd]_{m_2} \ar[dr]_{4:1}
& & P 
\ar[dd]^{\pi_{P}} \\
  &JH^* \times JH_y \times JH_z\ar[ur]^{\psi}   & \\
JH \times JH_y \times JH_z   \ar[rr]^-{p}& & P/G_P   
}
\end{equation}
where $p$ is an isomorphism of principally polarised abelian varieties  and the degrees of $\psi$, respectively
$\pi_P$, are $4^{2g-2}$, respectively  $4^{g-1}$.
\end{lem}
In particular,  $G_P = \Ker(\phi_{\Xi}) \cap  P[2]$. One also computes,
as in Lemma \ref{Intersection}, that $G_P = JH^* \cap JH_y \cap JH_z$. 
 Let $\mathcal{RH}_{g,8}$ denote the moduli space parametrising hyperelliptic Klein coverings over hyperelliptic curves of genus $g-1$ simply ramified in 8 points, so the covering curve is of genus $4g-3$.

\begin{thm} \label{mixed_Prym}
 For $g\geq 2$ the Prym map $$Pr^H_{4g-3,g-1}: \cRH_{g-1, 8} \ra \cA^{\delta}_{3g-1}, \qquad  ([\tC \ra H']) \mapsto (P(\tC, H'),
 \Xi)$$ is injective.
\end{thm}
\begin{proof}
The proof is very similar to that one of Theorem \ref{injectiveKlein}.
Consider $(P, \Xi)$ a polarised abelian variety in the image of 
$Pr^H_{4g-3,g-1}$. The subgroup $G:= P[2] \cap \Ker \phi_{\Xi}$ is isotropic and according to Lemma \ref{multtwo2} the quotient is isomorphic, as 
principally polarised abelian varieties, to a product  $JH_1\times JH_2 \times JH_3$ of uniquely determined Jacobians $JH_1, JH_2$ of dimension $g-1$ and one Jacobian $JH_3$, of dimension $g$.
 Without loss of generality and keeping the notation as above, we set $JH_1=JH_y$ and $JH_2= JH_z$ for some points $[y],[z] \in \PP^1$ and set $H:=H_3$.
 
 Consider the image $JH^*$ of $JH$ in $P$. The map $m_2: JH^* \ra JH^*$, multiplication by 2, factors through the restriction to $JH^*$ of the quotient map $\pi_P: P \ra P/G$, since $\Ker \pi  \subset \Ker m_2$.
\begin{equation} 
\xymatrix@R=1.1cm@C=.9cm{
JH^*   \ar[rr]^(.5){m_2} \ar[dr]_{\pi_P{_{|JH^*}}} 
& & JH^* \\
  &JH \ar[ur]_{4:1}^{\beta}   &
}
\end{equation}
Let $V_4:= \Ker \beta$, which is a subgroup of $JH$ generated by two 2-torsion points. 
Let $\pi_j : \tC \ra H_j$ the $4:1$ branched maps for $j\in\{y,z\}$ and 
$\pi_0: \tC \ra H$ the \'etale $4:1$ map (see Diagram \ref{diagetale}). Consider the map $k: \Pic^4(\tC) \ra \Pic^0(\tC)$  defined as before. So the 
maps $\alpha_j:=k \circ \pi_j^* $ are injective for $j\in\{x,y\}$ and 
$$
\alpha_0:=k \circ \pi_0^*: \Pic^1(H) \ra \Pic^0(\tC) \simeq J\tC
$$
has in its kernel 3 Weierstrass points (i.e. its pullback under $\pi_0$ is $ \sim 2g_2^1$), whose differences generate $V_4$. The argument in Proposition \ref{gluing} shows that 
$$
\alpha_y(\Pic^1 H_y) \cap \alpha_z(\Pic^1 H_z) \cap \alpha_0(\Pic^1 H) 
= JH_y \cap  JH_z \cap JH = G.
$$
Therefore, one can find a suitable automorphism of $\PP^1$ such that the images of the hyperelliptic maps from $H_j$ and $H$ consist of $2g-1$ points in $\PP^1$ and 
their Weierstrass points have $G$ as common image in $P \subset J\tC$. 
In order to determine the remaining point $[z]\in \PP^1$, we consider the two Weierstrass points 
$y,z\in H$ above $[y],[z]\in \PP^1$, whose images 
are not in $G$.  Then there is a unique point $x\in H$
(necessarily a Weierstrass point), 
such that $V_4=\langle \cO_H(x-y), \cO_H(y-z) \rangle$. 
One recovers the base curve $H':= H_x$ as the only hyperellitptic curve branched in $[x], [w_1], \ldots, [w_{2g-1}]\in \PP^1$. 
Applying the equivalence $(2) \Leftrightarrow (3)$ of Proposition \ref{equivalence1} one recovers the Klein covering $\tC \ra H'$ (see Figure \ref{fig:hyperelliptic_covers}).
\end{proof}

\begin{cor}
The following data are equivalent:
\begin{itemize}
\item[(1)] a triple $(W,W',B)$ of disjoint sets of points in $\PP^1$ such that $W$ is of cardinality 2g-1, $W'$ is a pair of points and $B$ is a point (up to a projective equivalence respecting the sets);
\item[(2)] a hyperelliptic genus $4g-3$ curve with a choice of a Klein subgroup of involutions $\langle \sigma,\iota\tau\rangle$, where $\s$, $\t$ are fixed point free;
\item[(3)] a hyperelliptic genus $g-1$ curve together with a 
Weierstrass points $x$ and two pairs  of points $y,\iota y$
and $z,\iota z$.
\end{itemize}
\end{cor}
 
\begin{rem}
Although from our point of view, the mixed case is less natural, it can be seen as the starting point of \cite{CMS}.
\end{rem}

\section{Branched $\ZZ_2^2$-coverings over hyperelliptic curves}
\subsection{From the top curve.} Let $\tC$ be a hyperelliptic curve of genus $4g+3$ admitting a subgroup of automorphisms
generated by three commuting involutions, namely  $\s, \t $ and the hyperelliptic involution $\iota$. By Proposition \ref{prop:cominv} 
we have 
$$
|\Fix (\t)| = |\Fix (\s)| = |\Fix (\iota\s\t)| = 0, \qquad |\Fix (\s\t)| = |\Fix (\iota\s)| = |\Fix (\iota\t)| = 4.
$$
For the convenience of the reader, we write $\alpha \in \{ \s,\t, \iota\s\t \}$ and $\beta \in \{ \iota\s,\iota\t, \s\t \}$.
By $C_{\alpha}=\tC/\al$, we denote the quotient curves of genus 2g+2 and $T_{\beta}$ the quotient curves of genus 2g+1. 
We have the following commutative diagrams
\begin{equation} \label{diag5}
\xymatrix@R=.9cm@C=.6cm{
&& \tC \ar[dr]^{et} \ar[dl]_{et} &&  \\
& C_{\t}\ar[dr]_{+2} \ar[dl]_{+2} & & C_{\s}\ar[dr]_{+2} \ar[dl]_{+2} & \\
H_{\t,\iota\s}\ar[drr] && H_{\s,\s\t} \ar[d] && H_{\s,\iota\t}\ar[dll] \\
&& \PP^1 &&
}
\qquad
\xymatrix@R=.9cm@C=.6cm{
& \tC \ar[dr]^{+4} \ar[dl]_{+4} \ar[d]^{+4} & \\
T_{\s \t}\ar[dr]_{+4} &  T_{\iota \t} \ar[d]_{+4} & T_{\iota\s} \ar[dl]^{+4}\\
&E \ar[d]& \\
&\PP^1&
}
\end{equation}
where $H_{\alpha,\beta}$ is the genus $g+1$ curve quotient of $\tC$ by the subgroup $\langle \alpha, \beta \rangle$ 
and $E$ is the quotient of $\tC$  by $\langle \iota \s, \iota \t \rangle$ which is the unique quotient curve of genus $g$ in the tower. Here $+4$, respectively 
$+2$, denotes a 2:1 map branched in 4, respectively 2 points.  By Corollary \ref{hypquot}, all the positive genus curves in both diagrams
are hyperelliptic.  

In order to describe the images of the Jacobians of the quotient curves in $J\tC$, we analyse the behaviour of the 2-torsion points under the pull-back maps. According to \cite[Prop 11.4.3]{bl}, $JT_\beta$ and $JE$ are embedded in $J\tC$ whereas the image of $JC_\al$ is not, so its image will be denoted by $JC_\al^*$.
Moreover, Diagram \eqref{diag6} shows that $JH_{\al,\beta}$ is not embedded in $J\tC$, so we will use the notation $JH_{\al,\beta}^*$ for its image.
\begin{equation} \label{diag6}
\xymatrix@R=.7cm@C=.7cm{
& \tC \ar[dr]^{+4} \ar[dl]_{et} & \\
C_\al \ar[dr]_{+2} & & T_{\beta} \ar[dl]^{et}\\
& H_{\al,\beta} &
}
\end{equation}

By construction, $JT_\beta=\im(1+\beta)$, $JE=\im(1-\s-\t+\s\t)$.
Moreover, we have that $JH^*_{\t,\iota\s}=\im(1-\s+\t-\s\t),\  JH^*_{\s,\s\t}=\im(1+\s+\t+\s\t),\ JH^*_{\s,\iota\t}=\im(1+\s-\t-\s\t)$.

One can easily compute that 
$$(1-\s+\t-\s\t)+(1+\s+\t+\s\t)+(1+\s-\t-\s\t)+(1-\s-\t+\s\t)=4$$
which shows that $J\tC=JE\boxplus JH^*_{\t,\iota\s}\boxplus JH^*_{\s,\iota\t}\boxplus JH^*_{\s,\s\t}$. 
As a result we get that $$P(\tC/E)=JH^*_{\t,\iota\s}\boxplus JH^*_{\s,\iota\t}\boxplus JH^*_{\s,\s\t}.$$

\subsubsection{2-torsion points on Jacobians}
We shall describe the images of the 2-torsion points in $J\tC$ of the Jacobians of the quotient curves.
We start by recalling well-known results concerning 2-torsion points on hyperelliptic Jacobians.
\begin{lem}
\label{2torsionlemma}
Let $W=\{w_1,\ldots,w_{2g+2} \}$ be the set of Weierstrass points on a hyperelliptic curve $C$ of genus $g$. Let $S=\{s_1,\ldots,s_{2k}\} \subset\{1,\ldots,2g+2\}$ be a set of even cardinality with $2k\leq g+1$. By $S^c$ we denote the complement of $S$ in $\{1,\ldots,2g+2\}$. We consider degree 0 divisors as elements of $JC=Pic^0(C)$ hence equality means linear equivalence of divisors. 
\begin{enumerate}
\item For all $i,j\leq 2g+2$, we have $w_i-w_j=w_j-w_i\in JC[2]$. Consequently, a divisor $D=\sum_{s\in S} \pm w_s$ does not depend on a particular choice of pluses and minuses, as long as its degree equals 0.
\item We have the  equality $\sum_{s_i\in S} \pm w_s=\sum_{t\in S^c} \pm w_t$, as long as degrees of both sides equal 0. 
\end{enumerate}
From now on, we will write $P_S=\sum_{i=1}^k(w_{s_{2i}}-w_{s_{2i-1}})$ to have degree 0 automatically and hence make notation more consistent. 
    \begin{enumerate}\setcounter{enumi}{2}
    \item If $g$ is even then for every $P\in JC[2], P\neq 0$ there exists a  unique $S$ of even cardinality with $|S|\leq g+1$ such that $P=P_S$. 
    \item If $g$ is odd then for every $P\in JC[2], P\neq 0$ there exists a  unique $S$ of even cardinality with $|S|\leq g$ such that $P=P_S$ or unique complementary pair $S,S^c$ of cardinality $g+1$ such that $P_S=P_{S^c}$.
\end{enumerate}
\end{lem}
\begin{proof}
A proof can be found in any of  \cite{D}. 
\end{proof}


In the case of $\tC$, we number the Weierstrass points as follows. Start with any Weierstrass point and call it $w_1$. Then we denote $\s w_1, \t w_1, \s\t w_1$ the other 3 Weierstrass points in the fibre of the map $\tC \ra H_{\sigma, \sigma\tau}$. Note that the four Weierstrass points are indeed different because $\s,\t$ are fixed point free and $\s\t$ have fixed points outside the set of Weierstrass points. Then proceed in the same way with the rest of the Weierstrass points. Since there are $2(4g+3)+2$ of them on $\tC$, in the end we will get \textit{the following numbering $W^{\tC}=\{w_1, \s w_1, \t w_1, \s\t w_1, w_2, \s w_2,\ldots,\t w_{2g+2},\s\t w_{2g+2}\}$ of Weierstrass points of $\tC$.} One has to be aware that we made a particular choice, however in Remark \ref{renumbering}, we will show that a different choice will only result in a permutation of indices and will not affect the results obtained.

For convenience, a Weierstrass point of $\tC$ when written in brackets will denote its image 
in the corresponding quotient curve. We start writing down 2-torsion points of $JE$. Since $E=\tC/\langle \iota\s,\iota\t\rangle $, on can easily check that $W^E=\{[w_1],\ldots,[w_{2g+2}]\}$. Setting  $e_i:=w_i+\s w_i+ \t w_i+ \s\t w_i$ and considering $JE$ as a subvariety of $J\tC$, one checks that 
\begin{equation}\label{generators}
\{ e_i-e_j :1\leq i<j\leq 2g+2\}\subset JE[2].
\end{equation}

Note that $e_i-e_j$ represents $[w_i]-[w_j]\in JE$, so Lemma \ref{2torsionlemma} shows that the sums of up to $\frac{g+1}{2}$ elements with disjoint indices give, on one hand,  different points of $J\tC$ and on the other, represent all possible 2-torsion points on $JE$. Therefore, we get that $JE[2]$ is generated by \eqref{generators}.

For the description of the 2-torsion points of $JT_\beta$, we will compute explicitly one case, namely $T:= T_{\s\t}=\tC/\s\t$. In this case, one checks that $W^T=\{[w_1],[\t w_1],\ldots,[w_{2g+2}],[\t w_{2g+2}]\}$. Again, denote by $v_i=w_i+\s\t w_i$ and $v_i'=\t w_i+\s w_i$, so the 2-torsion points of $JT$ are generated by $$\langle  v_i-v_j,v_i'-v_j,v_i-v_j', v'_i-v'_j :1\leq i<j\leq 2g+2\rangle$$
considered as embedded in $J\tC[2]$.
Indeed, using Lemma \ref{2torsionlemma} one checks that, by adding up to at most $2g+1$ generators, we generate all $2^{4g+2}$ 2-torsion points of the image of $JT$.

For the computation of $JH_{\al,\beta}^*[2]$ one has to take into account that half of the 2-torsion points on the quotient come from the 4-torsion points of $JH_{\al,\beta}$. We will compute explicitly the points on $JH^*_{\s,\s\t}[2]$ (and we will denote the curve by $H$). Note that $H$ is a quotient of $T$ (see Diagram \eqref{diag6}) hence we can use $v_i,v_i'$ to represent 2-torsion points on $JH$. Set $H:= H_{\s,\s\t}$.

Recall that the map $\pi:T\ra H=T/\s$ is an \'etale
double covering of hyperelliptic curves of genera $2g+1$ and $g+1$ as illustrated in Figure \ref{fig:branched_hyp_covers}. The set of Weierstrass points of $H$ equals $W^H=\{[w_1],\ldots,[w_{2g+2}],u_{2g+3},u_{2g+4}\}$ and
the covering is defined by a two torsion point $u_{2g+3}-u_{2g+4}\in JH[2]$. In particular, 
$\pi^{*}(u_{2g+3}-u_{2g+4})=0$.

Set $u_i:=v_i+v_i'$ and write the following subgroup of $JH^*[2]$:
$$U:=\left< u_i-u_j :1\leq i<j\leq 2g+2\right>.$$
The same computation as for $JE$ shows that the subgroup $U$ has $2^{2g}$ elements, so it is $25\%$ of all 2-torsion points of $JH^*$.
We can write the preimages $\pi^{-1}(u_{2g+3})=x+\iota x \text{ and } \pi^{-1}(u_{2g+4})=y+\iota y$, for some $x,y\in T$. 
Now, we are able to represent the preimage of $\pi^{*}([w_i]-u_{2g+3})$ as: $$\pi^{*}([w_i]-u_{2g+3})=u_i-g^1_2=v_i-v'_i\in JT.
$$
Note that, together with $U$, we can represent all these 2-torsion points as $\{P_S=\sum_{i\in S}v_i-v_i': |S|<g+1\}$, (where $S$ can be also of odd cardinality), so we get $2^{2g+1}$ points which represent $50\%$ of the 2-torsion points on $JH^*$, precisely the points that come from the 2-torsion points on $JH$.

The trickiest part is to represent 2-torsion points in $JH^*$ that come from 4-torsion points in $JH$, say $R$, satisfying $2R=u_{2g+3}-u_{2g+4}$. There are precisely $2^{2g+1}$ of them.
We will use the fact from Lemma \ref{2torsionlemma} that 
$$
u_{2g+3}-u_{2g+4}=\sum_{i\leq g+1}[w_{2i}]-[w_{2i-1}]\ \textnormal{as divisors in } JH.
$$
Now, $\pi^{-1}([w_{i}])=\{[w_i],[\t w_i]\}\subset T$. Since $v_i\in \Pic^2(\tC)$ represents $[w_i]$ and $v'_i$ represents $[\t w_i]$, we set $Q_i\in \{v_i,v'_i\}$
and  define $Q=\sum_{i\leq g+1}Q_{2i}-Q_{2i-1}$. There are precisely $2^{2g+2}$ of such representations with the relation 
$$
Q=Q' \quad \Leftrightarrow \quad \forall_j (Q_j=Q_j')\ \textnormal{ or } \ \forall_j (Q_j\neq Q'_j).$$ 
Therefore, we have defined $2^{2g+1}$ points in $JT[2]$ (seen as embedded in $J\tC[2]$). In order to show that indeed $\pi^{*}(R)=Q$ it is enough to note the following facts.
Firstly, the cardinalities of both sets (of possible Q's and possible $\pi^{*}(R)$'s) are equal to $2^{2g+1}$. Secondly, for every $Q$ we have that $\Nm_\pi(Q)=u_{2g+3}-u_{2g+4}$. Thirdly, by definition, we have that $\Nm_{\pi}\circ\pi^*=m_2$, where $m_2$ is the multiplication by 2, therefore
$$
\Nm_\pi(Q)=u_{2g+3}-u_{2g+4}=2R=\Nm_{\pi}(\pi^*(R))$$ 
and lastly, $\pi^*(R)$ are 2-torsion points and $Q$ are the only 2-torsion points of $JT$ with the property that $\Nm_\pi(Q)=u_{2g+3}-u_{2g+4}$.


In the following proposition we compile the analogous results for all other curves $JH_{\al,\beta}^*$.
\begin{lem}\label{descrip2tor}
We have the following generators of subgroups of 2-torsion points as embedded in $J\tC[2]$:
\begin{align*}
    JE[2]&=\left<(w_i+\s w_i+\t w_i+\s\t w_i)-(w_j+\s w_j+\t w_j+\s\t w_j):\ 1\leq i<j\leq 2g+2\right>,\\
    JH_{\s,\s\t}^*[2]&=JE[2]+\left<w_i+\s\t w_i-\t w_i-\s w_i: \ 1\leq i\leq 2g+2\right>+\\&+\left<\sum_{k\leq g+1}(Q_{2k}-Q_{2k-1}): Q_i\in\{w_i+\s\t w_i,\s w_i+\t w_i\}\right>, \\
    JH_{\t,\iota\s}^*[2]&=JE[2]+\left<w_i+\t w_i-\s w_i-\s\t w_i: \ 1\leq i\leq 2g+2\right>+\\&+\left<\sum_{k\leq g+1}(Q_{2k}-Q_{2k-1}): Q_i\in\{w_i+\s w_i,\t w_i+\s\t w_i\}\right>, \\
    JH_{\s,\iota\t}^*[2]&=JE[2]+\left<w_i+\s w_i-\t w_i-\s\t w_i: \ 1\leq i\leq 2g+2\right>+\\&+\left<\sum_{k\leq g+1}(Q_{2k}-Q_{2k-1}): Q_i\in\{w_i+\t w_i,\s w_i+\s\t w_i\}\right>, \\
    JT_{\s\t}[2]&=JE[2]+\left<(w_i+\s\t w_i)-(\s w_j+\t w_j), (w_i+\s\t w_i)-(w_j+\s\t w_j):\ 1\leq i,j\leq 2g+2\right>, \\
      JT_{\iota\s}[2]&=JE[2]+\left<(w_i+\s w_i)-(\s\t w_j+\t w_j), (w_i+\s w_i)-(w_j+\s w_j):\ 1\leq i,j\leq 2g+2\right>, \\
JT_{\iota\t}[2]&=JE[2]+\left<(w_i+\t w_i)-(\s\t w_j+\s w_j), (w_i+\t w_i)-(w_j+\t w_j):\ 1\leq i,j\leq 2g+2\right>.
\end{align*}
\end{lem}
\begin{proof}
The proof is completely analogous to what we have done for $JH^*_{\s,\s\t}$. One only needs to change the subscripts depending on the involution by which it is divided. For 
example, for $H_{\s,\iota\t}$, the only involution with fixed points is $\iota\t$, so $Q_i\in\{w_i+\t w_i,\s w_i+\s\t w_i\}$.
\end{proof}


\begin{rem}\label{renumbering}
Note that a different choice  of Weierstrass points $w_1,\ldots,w_{2g+2}\in \tC$ will only result in a permutation of indices because all sets of generators are invariant under $\langle \s,\tau\rangle$. In particular, the statement of Lemma \ref{descrip2tor} does not depend on the numbering of the Weierstrass points we started with.
\end{rem}

\subsection{From the bottom curve}
In order to construct a $\ZZ_2^2$-branched covering of hyperelliptic curves, one only needs to choose  2g+5 points in $\PP^1$ with a distinguished triple, denoted by $[w_1],\ldots,[w_{2g+2}], [x],[y],[z]$. Then, the curve $E$ is  the hyperelliptic curve that is a double cover of $\PP^1$ branched in $[w_1],\ldots,[w_{2g+2}]$ points. The curve $T_x$ will be the double cover of $E$ branched at $y_E,\iota y_E, z_E,\iota z_E$ with the corresponding line bundle being the hyperelliptic $g^1_2$. According to Proposition \ref{hyp-cov},  $T_x$ is hyperelliptic. Then the preimages of $x_E$ become $x_T,x'_T$ and therefore one chooses $x_T,x'_T,\iota x_T, \iota x'_T$ as a branching for the map $\tC\ra T_x$ with the line bundle being $g^1_2$. By construction (and Proposition \ref{hyp-cov}), $\tC$ is a hyperelliptic curve of genus $4g+3$. The covering involution of the map $T_x\ra E$ lifts to $\tC$, hence $\tC$ is a hyperelliptic curve with two (additional) commuting involutions.

Moreover, the construction is uniquely defined up to a projective equivalence. To see this, one can directly construct the curve $\tC$ as follows. Consider a $\ZZ_2\times\ZZ_2$-covering of $\PP^1$ branched in $[x],[y],[z]$ with two simple ramifications on each fiber 
(the existence is shown in Remark \ref{p1}). By Hurwitz formula the covering curve is of genus 0. Moreover, there are $8g+8$ points over $[w_1],\ldots,[w_{2g+2}]$ and there is an action of $\ZZ_2\times\ZZ_2$ on them.
Then $\tC$ is constructed as the double cover branched in these $8g+8$  points. 
Since the set of branching points is invariant under the action of $\ZZ_2\times\ZZ_2$ the curve $\tC$ possesses two commuting involutions. In this way, we constructed all the maps of the following commutative diagram.
\begin{equation} \label{Const_tildeC}
\xymatrix@R=.9cm@C=1cm{
 \tC \ar[d]_{2:1} \ar[r]^{4:1} & E \ar[d]^{2:1} \\
\PP^1 \ar[r]^{4:1} & \PP^1 
}\end{equation}

\begin{rem}\label{p1}
An example of such a covering of the projective line can be described as $[x:y]\ra[x^4+y^4:2x^2y^2]$ with the branching points $[1:0],[1:-1],[1:1]$. The deck transformations are explicitly described as $[x:y]\ra[-x:y]$ and $[x:y]\ra [y:x]$.
\end{rem}

We finish this section by showing the equivalence of data needed to built a branched $\ZZ_2^2$ covering.

\begin{prop} \label{equivalence}
For $g\geq 1$, the following data are equivalent:
\begin{enumerate}
    \item $2g+5$ points in $\PP^1$ with a distinguished triple up to a projective transformation;
    \item a genus $4g+3$ hyperelliptic curve $\tC$ with 2 commuting involutions;
    \item a genus $g$ hyperelliptic curve $E$ with three pairs of points $x,\iota x, y, \iota y, z,\iota z$ up to an isomorphism;
    \item 3 genus $g+1$ hyperelliptic curves $H_x$, $H_y, H_z$ that have branching points that can be glued to get a set of $2g+5$ points with each of 3 distinguished points shared by precisely 2 curves. 
\end{enumerate}
\end{prop}
\begin{proof}
The equivalence $1\iff3$ is given by the hyperelliptic covering. The equivalence $2\iff 3$ follow from the construction of branched $\ZZ^2_2$ coverings seen from top and from bottom.
The implication $2\Rightarrow 4$ comes from taking 3 quotient curves by 3 Klein subgroups (see Diagram \eqref{diag5}) and $4\Rightarrow 1$ is obvious.
\end{proof}

\begin{figure}[h!] \label{fig:branched_hyp_covers}
\includegraphics[width=11.5cm,  trim={0cm 7cm 0cm 5cm}, clip]{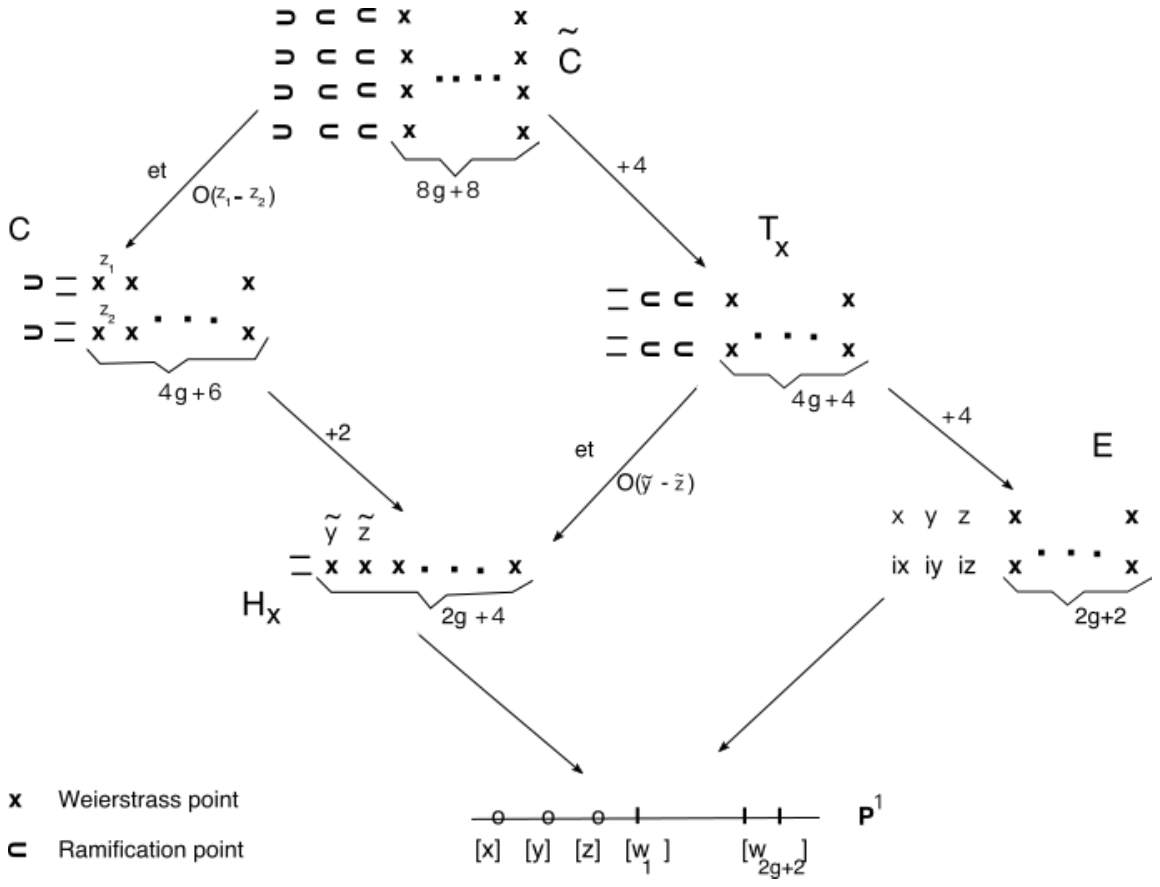}
  \caption{ Weierstrass and ramification points on hyperelliptic branched coverings}
\end{figure}

\subsection{Prym map}
The aim of this section is to show that the Prym map for hyperelliptic Klein coverings branched in 12 points is injective.
In our construction the pullback of $E$ under the 4:1 map $\tC \ra E$ defines a  subvariety with polarisation of type $(4, \ldots, 4)$ therefore, as a complementary subvariety of $E$ in $J\tC$, the Prym variety $P$ corresponding to the map $\tC \ra E$,
has a  polarisation $\Xi$ of type 
$$
\delta:= (\underbrace{1,\ldots,1}_{2g+3}, \underbrace{4, \ldots, 4}_{g}).$$
Let $\cRH_{g,12}$ denote the moduli space of pairs $(E, \{q_1, q_2, q_3\})$ where $E$ is a hyperelliptic curve and
the points $\{q_1, q_2, q_3\} \subset E $ are not pairwise conjugated. In view of Proposition \ref{equivalence} this moduli space parametrises also the hyperelliptic $\ZZ_2^2$-
coverings branched in 12 points.  
\begin{prop}
The moduli space $\cRH_{g,12}$ is irreducible.
\end{prop}
\begin{proof}
It follows from the equivalence $(1) \Leftrightarrow (3)$ of the Proposition \ref{equivalence}.
\end{proof}

We consider the Prym map 
$$
Pr^H_{4g+3,g}: \cRH_{g,12} \ra \cA_{3g+3}^{\delta}, \qquad (E, \{q_1, q_2, q_3\}) \mapsto (P(\tC/ E) , \Xi).
$$
Observe that the image of this Prym map is contained in an irreducible component of $\cA_{3g+3}^{\delta}$ whose 
elements admit a $\ZZ_2 \times \ZZ_2 $ automorphism subgroup acting on them and leaving invariant the 
algebraic class of the polarisation $\Xi$. 


Recall that the Prym variety of $\tC\ra E$ decomposes as $P(\tC/E)=JH^*_{\t,\iota\s}\boxplus JH^*_{\s,\iota\t}\boxplus JH^*_{\s,\s\t}$.
\begin{lem}\label{lem58}
We have the equality as subgroups of $J\tC$
\begin{eqnarray*}
JH^*_{\t,\iota\s}\cap JH^*_{\s,\iota\t}\cap JH^*_{\s,\s\t} &=& JE[2]+\left<w_i+\t w_i-\s w_i-\s\t w_i: \ 1\leq i\leq 2g+2\right> \\ &=& 
JE[2]+\ZZ_2(w_1+\t w_1-\s w_1-\s\t w_1).
\end{eqnarray*}
In particular, the intersection is of order $2^{2g+1}$. 
\end{lem}
\begin{proof}
According to Lemma \ref{2torsionlemma} (1) and Lemma \ref{descrip2tor}, we have the inclusion from the right hand side of the first equality. To prove the equality one uses the 
description of Lemma \ref{descrip2tor} and check the elements 
of the form $\sum_{k\leq g+1}(Q_{2k}-Q_{2k-1})$ can not be contained in the intersection. Indeed, there are precisely $2g+2$ summands, so two such divisors are linearly equivalent if and only if all summands coincide or the summands are complementary. Since the involutions involved are different, for different curves some summands (but not all) will be the same.
The second equality follows from the fact that one can write 
\begin{eqnarray*}
(w_i+\t w_i+ \s w_i + \s\t w_i) -  (w_1+\t w_1+ \s w_1 + \s\t w_1) + && \\
w_1+\t w_1-\s w_1-\s\t w_1 &=&  (w_i+\t w_i+ \s w_i + \s\t w_i) -2\s w_1 -2\s\t w_1  \\ &\sim&
(w_i+\t w_i+ \s w_i + \s\t w_i) -2\s w_i -2\s\t w_i \\
&=& (w_i+\t w_i - \s w_i - \s\t w_i).
\end{eqnarray*}
\end{proof}
\begin{rem}
Note that for $Q_i\in\{w_i+\s w_i,\t w_i+\s\t w_i\}, R_i\in\{w_i+\t w_i,\s w_i+\s\t w_i\}$ we have
$$\sum_{k\leq g+1}(Q_{2k}-Q_{2k-1})+\sum_{k\leq g+1}(R_{2k}-R_{2k-1})\in JH^*_{\s,\s\t}.$$
This relation gives many of the elements in the kernel of the addition map $JH^*_{\t,\iota\s}\times JH^*_{\s,\iota\t}\times JH^*_{\s,\s\t}\to P(\tC/E)$.
\end{rem}

Now, we would like to use a modified version of \cite[Prop 3.1]{NOS}. 
\begin{prop}\label{newprop}
\label{group}
    Let $(P, \Xi)$ be an element of $\im(Pr^H_{4g+3,g})$. Then the group of automorphisms $$\{ \gamma \in \Aut(P,\Xi) \: | \: \: \gamma(x) = x, \:\: \forall x \in K(\Xi)\}$$
    is isomorphic to $\langle \iota\sigma, \iota\tau \rangle \simeq \mathbb{Z}_2 \times \mathbb{Z}_2$
\end{prop}

\begin{proof}
    Denote by $f: \tC \to E$ a hyperelliptic Klein covering with $P(\tC/E) = P(f)= P$ and Galois group $\langle \iota\sigma, \iota\tau \rangle \simeq \mathbb{Z}_2 \times \mathbb{Z}_2$. Note that $K(\Xi) = f^*JE \cap P(f) \subset \Fix(\tau) \cap \Fix(\sigma)$. Thus, there is an automorphism $\Tilde{\gamma}: J\tC \to J\tC$ such that the following diagram commutes:

\begin{equation} \label{55}
\xymatrix@R=.9cm@C=.9cm{
0\ar[r]^(.7){}&K(\Xi)\ar[r]^(.7){}\ar[d]_{=}&f^*JE \times P  \ar[r]^(.7){\mu} \ar[d]_{(id,\gamma)} & J\tC 
\ar[d]^{\Tilde\gamma} \ar[r]^(.7){}&0\\
0\ar[r]^(.7){}&K(\Xi)\ar[r]^(.7){}&f^*JE \times P    \ar[r]^-{\mu}&  J\tC   \ar[r]^(.7){}&0
}
\end{equation}
where $\mu$ is the addition map. Since $\gamma$ is a polarised isomorphism, we get from Diagram \ref{55} that $\mu^*\Tilde{\gamma}^*\mathcal{O}_{J\tC}(\Tilde{\Theta})$ and $ \mu^*\mathcal{O}_{J\tC}(\Tilde{\Theta})$ are equal as polarisations.

Since $\mu^*$ has a finite kernel, we have that 
$\Tilde{\gamma}^*\mathcal{O}_{J\tC}(\Tilde{\Theta}) \otimes \mathcal{O}_{J\tC}(\Tilde{\Theta})^{-1}$ is a torsion sheaf, hence it belongs to $\Pic^0(J\tC)$. Therefore, $\Tilde{\gamma}^*\mathcal{O}_{J\tC}(\Tilde{\Theta})$ induces the canonical principal polarisation on $J\tC$ and  $\tg$ is a polarised isomorphism. 
Since $\tC$ is hyperelliptic, by the strong Torelli Theorem \cite[Ex 11.12.19]{bl}, there is an automorphism $\tg_{\tC}$ of $\tC$ inducing $\tg$. 

Now, since $f$ is branched,  $f^*$ is an embedding of $JE$ in $J\tC$ and by construction $\tg_{|JE}=id$. Hence 
$\tg_{\tC}$ is a lift of $id_{E}$, so it lies in the group of deck transformations of the covering $f$ that is equal to $\langle \iota\sigma, \iota\tau \rangle$. This shows that $\gamma \mapsto \tg_{\tC}$ gives a desired isomorphism, because the inverse map is just the restriction of a deck transformation of $J\tC$ to $P$.
\end{proof}
\begin{rem}
    Note that $K(\Xi)$ contains 4-torsion points, so $-1$ does not belong to the group fixing $K(\Xi)$. Our result is stronger than \cite[Prop 3.1]{NOS}  because we have shown the isomorphism for all Prym varieties and not only for the general ones. Here have we used the fact that the top curve is hyperelliptic and the covering is branched.
\end{rem}
\begin{cor}\label{newcor}
    Note that $P=JH^*_{\t,\iota\s}\boxplus JH^*_{\s,\iota\t}\boxplus JH^*_{\s,\s\t}$ is the isotypical decomposition for the group defined in Proposition \ref{newprop}.
\end{cor}

Now, we are ready to prove the main result of this section.
\begin{thm} \label{brached_inj_Prym}
For $g>0$, the Prym map $Pr^H_{4g+3,g}$ is injective.
\end{thm}
\begin{proof}
Let $P\in\im(Pr^H_{4g+3,g})$. By Proposition \ref{newprop} we can construct the Klein four-group acting on $P$ and we can perform the isotypical decomposition to obtain three abelian subvarieties uniquely determined by the action of Klein group on $P$, called $A,B,C$ (see Corollary \ref{newcor}). Let $G=A\cap B\cap C$ and note that by Lemma \ref{lem58} the cardinality of $G$ is $2^{2g+1}$. 

Set $A/G=:JH_x$, $B/G=:JH_y$, $C/G=:JH_z$. Since $G$ contains only 2-torsions, we can extend the quotient map to the map $A\ra A/G\ra A$ such that the composition is multiplication by $2$. Moreover, since $G$ is order $2^{2g+1}$ we get that $A/G=JH_x\ra A$ is of order 2, hence given by a $2$ torsion of the form $w^{H_x}_{2g+3}-w^{H_x}_{2g+4}\in \ker JH_x\ra A$.
We denote the remaining Weierstrass points of $H$ by $w^{H_x}_1,\ldots,w^{H_x}_{2g+2}$. By taking the images under the hyperelliptic covering, we get the points $[w_1],\ldots,[w_{2g+2}]\in\PP^1$.

Similarly to the \'etale case, for $j\in\{x,y,z\}$, we can define 
$\alpha_j=k\circ\pi^*_j:\Pic^1(H_j)\ra J\tC$, (where $k(D)=D-2g^1_2$) although
note that $\alpha_j$ is of degree 2, since $\pi^*(w_{2g+3})=\pi^*(w_{2g+4})=2g^1_2$.
However it is still true that $\alpha_j(w_p)\neq\alpha_j(w_{q})$ for $p\neq q<2g+3$.
Therefore, we can renumber the Weierstrass points of $H_y$ in such a way that $\alpha_x(w^{H_x}_i)=\alpha_y(w^{H_y}_i)$ for $i=1,\ldots,2g+2$.

This compatibility allows us to show that having the hyperelliptic covering of $H_y$ there exists an automorphism of $\PP^1$ such that images of the Weierstrass points  coincide, i.e. $[w^{H_y}_i]=[w^{H_x}_i]$, $i=1,\ldots,2g+2$. Since $g\geq 1$, this automorphism is unique. 
Moreover, by construction, we get that $[w^{H_x}_{2g+3}]=[w^{H_y}_{2g+3}]=:[z]$ and  $[w^{H_x}_{2g+4}]=[y]$, $[w^{H_y}_{2g+4}]=[x]$ are distinct.

We can perform a similar argument for $H_z$ to get the unique automorphism of $\PP^1$ such that the images of Weierstrass points of $H_z$ become $\{[w^{H_x}_{1}],\ldots,[w^{H_x}_{2g+2}], [x], [y]\}$.

Note that, although in the construction we have used $\alpha_j$ that are a priori defined for a chosen $\tC$, in fact we only need the equality of images of the Weierstrass points that lie in $P$, so the construction is intrinsic.
In this way, we have constructed a unique set of $2g+5$ points of $\PP^1$ with a distinguished triple (up to projective equivalence). This proves the injectivity of $Pr^H_{4g+3,g}$.
\end{proof}
\begin{rem}
Note that, unlike the \'etale case, a simple computation of degrees  shows that $P\ra P/G$ cannot be a polarised isogeny, so we need to divide each subvariety individually.
\end{rem}

\subsection{The mixed case of $4g+3$}
From up to bottom perspective, this case occurs when one starts with a genus $4g+3$ curve with two fixed point free involutions $\s,\t$ (with $\s\t$ having 4 fixed points) and takes a group generated by $\langle \s,\t\rangle$.
The tower of curves can be found in Diagram \ref{diag5} when we treat $H_{\s,\s\t}$ as the base curve.
From what we have already described, one can easily deduce that in this case, the Prym variety $P(\tC/H_{\s,\s\t})=JE\oplus JH^*_{\t,\iota\s}\oplus JH^*_{\s,\iota\t}$. Since most of the computations have already been done, we will focus on stating the results and main steps.

For $\delta=(\underbrace{1,\ldots,1}_{2g+1},\underbrace{2,4,\ldots,4}_{g+1})$, define the Prym map:
$$Pr ^H_{4g+3,g+1}:\cRH_{g+1,4}\lra\cA^\delta_{3g+2}.$$
\begin{thm}\label{mixed4g+3}
For $g>0$, the Prym map $Pr^H_{4g+3,g+1}$ is injective.
\end{thm}
\begin{proof}
Firstly, for $P\in\im(Pr^H_{4g+3,g+1})$ we have a similar result as  Proposition \ref{newprop} in this case, so analogous to  Corollary \ref{newcor} we can distinguish three abelian subvarieties of $P$ appearing in the isotypical decomposition. One of the subvarieties is of dimension $g$ denoted by $JE$ and the other two by $B$ and $C$. 

By Lemma \ref{descrip2tor}, we note that $JE\cap B\cap C=JE[2]$ is of order $2^{2g}$ and $B\cap C$ is of order $2^{2g+1}$.
Denote by $w_1,\ldots,w_{2g+2}$ the Weierstrass points of $E$ and $[w_1],\ldots,[w_{2g+2}]$ their images in $\PP^1$.

Using results from the proof of Theorem \ref{brached_inj_Prym}, by taking $G=B\cap C$, we see that  there exist unique curves $H_{y}, H_{z}$ such that 
$B=JH^*_{y},\ C=JH^*_{z}$ with the quotient maps given by the differences $w^{H_j}_{2g+3}-w^{H_j}_{2g+4}$ for $j=y,z$.

As before, consider the pullback map $\pi_E^*:\Pic^1(E)\ra \Pic^4(\tC)$ and $k:\Pic^4(\tC)\ra J\tC$ given by $k(D)=D-2g^1_2$. Then $k\circ \pi_E^*:\Pic^1(E)\ra J\tC$ is a monomorphism and the map $\alpha_j:=k\circ \pi_{JH_j}^*$ is of degree 2
for $j=y,z$. Note that we can number the Weierstrass points on $H_x,H_y$ using the condition $\alpha_j(w^{H_j}_l)=\alpha_E(w_l)$ for $l=1,\ldots,2g+2$ and by the fact that we are in the Prym locus we  get that, out of four points $w^{H_j}_{2g+3}, w^{H_j}_{2g+4}$ for $j=y,z$, precisely two have the same projection to $\PP^1$ denoted by $[x]$, and we denote the image in $\PP^1$ of the other two by $[y],[z]$ respectively. 
To summarise, starting from the Prym variety $P$, we have constructed $2g+5$ points in $\PP^1$ with a chosen triple $x,y,z$ and a distinguished point $x$ that yields the Prym variety we started with.

In this way, we have proved that the Prym map $Pr^H_{4g+3,g+1}$ has an inverse, hence it is injective.
\end{proof}

\begin{cor}
In the process, we have shown that the following data equivalent.
\begin{itemize}
\item[(1)] a triple $(W,W',B)$ of disjoint sets of points in $\PP^1$ such that $W$ is of cardinality 2g+2, $W'$ is a pair of points and $B$ is a point 
(up to a projective equivalence respecting the sets),
\item[(2)] a hyperelliptic genus $4g+3$ curve with a choice of a Klein subgroup of involutions $\langle \s,\t\rangle$ such that $\s$ and $\t$ are fixed point free and $\s\t$ has 4 fixed points,
\item[(3)] a hyperelliptic genus $g+1$ curve together with a pair of Weierstrass points $y,z$ and a pair of points $x,\iota x$.
\end{itemize}
\end{cor}

\section{Final remarks}
We assumed $g\geq 2$ in the \'etale case, because $g=1$ gives a trivial Prym and $g=0$ is impossible.
In the branched case, we assumed $g\geq 1$. For $g=0$ one gets that the Prym variety is the whole Jacobian. However, it must be noted that the mixed case is 'non-trivial' and the proof of Theorem \ref{mixed4g+3} does not work for $g=0$. This case has been investigated in a joint paper of the first author with Anatoli Shatsila where they have shown that the Prym map is generically of degree 2, see \cite{BS}.

\begin{rem}
We would like to point out that throughout the paper we used coordinate-free point of view. However, one can also work with equations. It can be checked that a hyperelliptic curve given by $y^2=(x^4+a_1x^2+1)\cdot\ldots\cdot(x^4+a_nx^2+1)$ has two additional commuting involutions given by $(x,y)\mapsto (-x,y)$ and $(x,y)\mapsto(\frac{1}{x},\frac{y}{x^{2n}})$ (see \cite{Sh}). 
This curve is of genus $2n-1$ and since the family depends on $n$ parameters, one can use Propositions \ref{equivalence1} and \ref{equivalence} to show that a hyperelliptic Klein covering can be given by such an equation.
\end{rem}


\end{document}